\documentclass[a4paper,12pt]{amsart}
 \usepackage[arrow,matrix]{xy}
\usepackage{amsmath,amssymb,amscd,bbm,amsthm,mathrsfs,dsfont,chemarrow}
\usepackage{graphicx}
 \numberwithin{equation}{section}
\usepackage{CJK}
\newtheorem{Theorem}{Theorem}[section]
\newtheorem{Lemma}{Lemma}[section]
\newtheorem{Corollary}{Corollary}[section]
\newtheorem{Proposition}{Proposition}[section]

\theoremstyle{definition}

\theoremstyle{remark} \theoremstyle{example}

\newtheorem{Remark}{Remark}[section]
\usepackage{geometry}
\begin{document}\numberwithin{equation}{section}

\title [ Kropina metrics with isotropic scalar curvature]{Kropina metrics with isotropic scalar curvature }

\author{ Liulin Liu, Xiaoling Zhang* and Lili Zhao}

\begin{abstract}
In this paper, we study Kropina metrics with isotropic scalar curvature. First, we obtain the expressions of Ricci curvature tensor and scalar curvature. Then, we characterize the Kropina metrics with isotropic scalar curvature on by tensor analysis.
\end{abstract}
 \maketitle

\footnote{
Supported partially by the National Natural Science Foundation of China (Grant Nos. 11961061, 11461064, 12071283).\\
*corresponding author.}

\section{introduction}
The curvature properties of metrics play very important roles in Riemannian and Finsler geometry. Riemannian curvature and Ricci curvature are the most important Riemannian geometric quantities in Finsler geometry.
 In 1988, the concept of Ricci curvature was first proposed by Akbar-Zadeh and its tensor form can be naturally obtained \cite{AZ}.
 In recent years, many scholars have done a lot of research on them.
 Cheng-Shen-Tian proved that polynomial $(\alpha,\beta)$-metric is an Einstein metric if and only if it is Ricci-flat \cite{CST}.
 Zhang-Shen gave the expression of Ricci curvature of Kropina metric. And they proved that a non-Riemannian Kropina metric with constant Killing form $\beta$ is an Einstein metric if and only if $\alpha$ is also an Einstein metric \cite{ZXL}.
 By using navigation date $(h, W)$, they proved that $n$$(\geq 2)$-dimensional Kropina metric is an Einstein metric if and only if Riemann metric $h$ is an Einstein metric and $W$ is a Killing vector field about $h$.
 Xia gave the expression for the Riemannian curvature of Kropina metrics and proved that Kropina metric is an Einstein metric if and only if it is of a non-negative constant flag curvature \cite{XQL}.
 Cheng-Ma-Shen studied and characterized projective Ricci flat Kropina metrics and obtained its equivalent characterization equation \cite{CMS}.

Unlike the notion of Riemannian curvature, there is no unified definition of scalar curvature in Finsler geometry, although several geometers have offered several versions of the definition of the Ricci curvature tensor\cite{AZ,L-S,16,25}. In 2015, Li-Shen introduced a new definition of the Ricci curvature tensor \cite{L-S}. This tensor is symmetric. They proved that for a Finsler metric $F$, it is of isotropic Ricci curvature tensor if and only if it is of isotropic Ricci curvature and $\chi$-curvature tensor satisfies $\chi_i=f_{ij}(x)y^j$, where $f_ {ij}+f_{ji}=0$. It is further proved that for Randers metrics, it is of isotropic Ricci curvature tensor if and only if it is of isotropic Ricci curvature.

In Finsler geometry, there are several versions of the definition of scalar curvature. We used Akbar-Zadeh's definition \cite{AZ} of the scalar curvature, based on the Li-Shen's definition of the (symmetric) Ricci curvature tensor \cite{L-S}. For a Finsler metric $F$ on an $n$-dimensional manifold $M$, the scalar curvature $R$ of $F$ is defined as $R:=g^{ij}Ric_{ij}$.
Tayebi studied general fourth-root metrics \cite{T}. They proved the necessary and sufficient condition for general fourth-root metrics with isotropic scalar curvature and found the necessary and sufficient condition for general fourth-root metrics with isotropic scalar curvature under conformal variation. Finally, they characterized Bryant metric with isotropic scalar curvature.
Chen-Xia studied a conformally flat $(\alpha,\beta)$-metric with weakly isotropic scalar curvature \cite{CX}. They proved that if  conformally flat polynomial $(\alpha,\beta)$-metrics have weakly isotropic scalar curvature $R$, then $R$ vanishes.

In this paper, we obtain the necessary and sufficient condition of Kropina metrics with isotropic scalar curvature and have the following results.
\begin{Theorem}\label{miR}
Let $F$ be a Kropina metric on an $n$$(\geq3)$-dimensional manifold $M$. Then $F$ is of isotropic scalar curvature if and only if
\begin{equation}\label{condition1}
\left\{
\begin{array}{lcl}
\displaystyle b^kb^l{}^\alpha Ric_{kl}=\frac{b^2}{2(n-1)}{}^\alpha R-\frac{(n-2)}{2(n-1)}s^t_{\;m}s^m_{\;t}-\frac{(n-2)(n+1)}{2(n-1)b^2}s^ms_m-\frac{(n-2)}{2}(2c_b-c^2),\vspace{1ex}\\
\displaystyle r_{00}=c(x)\alpha^2,\vspace{1ex}\\
\displaystyle f\alpha^2=-b^4{\;}^\alpha Ric+(n-2)[c^2\beta^2+(2cs_0-b^2c_0)\beta+(s_0^2-b^2s_{0|0})],
\end{array}\right.
\end{equation}
where $f$ is expressed by (\ref{2-4}). In this case, scalar curvature is 
$$R=-\frac{n}{4b^2}(2s^ms_m+b^2s^t_{\;m}s^m_{\;t}).$$
\end{Theorem}
\section{Preliminaries}
Let $M$ be an $n$-dimensional $C^\infty$ manifold. A Finsler structure of $M$ is a function $$F: TM \rightarrow[0, \infty)$$ with the following properties:

(1) Regularity: $F$ is $C^{\infty}$ on the slit tandent bundle $TM\backslash\{0\}$;

(2) Positive homogeneity: $F(x,\lambda y)=\lambda F(x,y)$, $\forall\lambda>0$;

(3) Strong convexity: 
$$g_{ij}(x,y):=\frac{1}{2}\frac{\partial^2F^2}{\partial y^i\partial y^j}(x,y)=\frac{1}{2}(F^2)_{y^iy^j}$$
is positive-definite at every point of $TM\backslash\{0\}$.

Let $(M,F)$ be an $n$-dimensional Finsler manifold. Suppose that $x \in M$. The geodesics of a Finsler metric $F=F(x, y)$ on $M$ are classified by the following ODEs:
$$
\frac{d^{2} x^{i}}{d t^{2}}+2 G^{i}\left(x, \frac{d x}{d t}\right)=0,
$$
where
$$
G^{i}:=\frac{1}{4} g^{i k}\left[\left(F^{2}\right)_{x^{j}y^{k}} y^{j}-\left(F^{2}\right)_{x^{k}}\right],
$$
$\left(g^{i j}\right):=\left(g_{i j}\right)^{-1}$. The local functions $G^{i}=G^{i}(x, y)$ are called geodesic coefficients (or spray coefficients).
Then the $S$-curvature, with respect to a volume form $dV=\sigma(x)dx$, is defined by $$S=\frac{\partial G^m}{\partial y^m}-y^m\frac{\partial\ln \sigma}{\partial x^m}.$$

For $x \in M$ and $y \in T_x M \backslash \{0\}$, Riemann curvature $R_{y}:=R^{i}_{\;k}(x,y) \frac{\partial}{\partial x^{i}} \otimes dx^{k}$ is defined by
$$R_{\;k}^{i}:=2 \frac{\partial G^{i}}{\partial x^{k}}-y^{j} \frac{\partial G^{i}}{\partial x^{j} \partial y^{k}}+2 G^{j} \frac{\partial G^{i}}{\partial y^{j} \partial y^{k}}-\frac{\partial G^{i}}{\partial y^{j}} \frac{\partial G^{j}}{\partial y^{k}}.$$
The trace of Riemann curvature is called Ricci curvature of $F$, i.e., $Ric:=R^k_{\;k}$.

 Riemann curvature tensor is defined by $$R^{{\ }\;i}_{j\;kl}=\frac{1}{3}\left\{\frac{\partial^2 R^i_{\;k}}{\partial y^l\partial y^j}-\frac{\partial^2 R^i_{\;l}}{\partial y^k\partial y^j}\right\}.$$
Let $\overline{Ric}_{ij}:=R^{{\ }\;k}_ {j\;kl}$. And
$$Ric_{ij}:=\frac{1}{2}\left\{\overline{Ric}_{ij}+\overline{Ric}_{ji}\right\}$$ is called Ricci curvature tensor.
The scalar curvature $R$ of $F$ is defined by $$R:=g^{ij}Ric_{ij}.$$
Let $\kappa(x)$ be a scalar function on $M$, $\theta:=\theta_i(x)y^i$ be a $1$-form on $M$. If $$R=n(n-1)\left[\frac{\theta}{F}+\kappa(x)\right],$$ then it is called that $F$ is of weak isotropic scalar curvature. Especially, when $\theta=0$, i.e., $R=n(n-1)\kappa(x)$, it is called that $F$ is of isotropic scalar curvature.

Let $F$ be  a Finsler metric on $M$. If $F=\frac{\alpha^2}{\beta}$, where $\alpha=\sqrt{a_{ij}(x)y^iy^j}$ is a Riemannian metric, $\beta=b_i(x)y^i$ is a $1$-form, then $F$ is a Kropina metric.
Its fundamental tensor $g=g_{ij}dx^i\otimes dx^j$ is given by \cite{XQL}
\begin{align*}
g_{ij}=\frac{F}{\beta}\left\{2a_{ij}+\frac{3F}{\beta}b_ib_j-\frac{4}{\beta}(b_iy_j+b_jy_i)+\frac{4y_iy_j}{\alpha^2}\right\},
\end{align*}
where $y_i:=a_{ij}y^j$.
Moreover,
\begin{align*}
g^{ij}=\frac{\beta}{2F}\left\{a^{ij}-\frac{b^ib^j}{b^2}+\frac{2}{b^2F}(b^iy^j+b^jy^i)+2\left(1-\frac{2\beta}{b^2F}\right)\frac{y^iy^j}{\alpha^2}\right\},
\end{align*}
 where $(a^{ij}):=(a_{ij})^{-1}$, $b^i:=a^{ij}b_{j}$.

 Let $\nabla\beta=b_{i|j}y^idx^j$ denote the covariant derivative of $\beta$ with respect to $\alpha$. Set
\begin{align*}
    &r_{ij}=\frac{1}{2}\left(b_{i|j}+b_{j|i}\right),\ s_{ij}=\frac{1}{2}\left(b_{i|j}-b_{j|i}\right),\ r_{00}=r_{ij}y^{i}y^{j},\ r_{i}=b^{j}r_{ij},
    \ r_{0}=r_{i}y^{i},\\
    &r^{i}=a^{ij}r_{j},\ r=b^{i}r_{i},\ {s^{i}}_{0}=a^{ij}s_{jk}y^{k},\ s_{i}=b^{j}s_{ji},
    \ s_{0}=s_{i}y^{i},\ s^{i}=a^{ij}s_{j}.
\end{align*}

The Ricci curvature of Kropina metrics is given by the following.
\begin{Lemma}[\cite{ZXL}]\label{Ric}
Let $F$ be a Kropina metrics on $M$. Then the Ricci curvature of $F$ is given by
$$Ric={}^\alpha Ric+T,$$
where ${}^\alpha Ric$ is the Ricci curvature of $\alpha$, and
$$\begin{aligned}
 T=&\frac{3(n-1)}{b^4 \alpha^4} r_{00}^2 \beta^2+\frac{n-1}{b^2 \alpha^2} r_{00 \mid 0}\beta -\frac{4(n-1)}{b^4 \alpha^2}r_{00} r_0 \beta +\frac{2(n-1)}{b^4 \alpha^2} r_{00} s_0\beta-\frac{r}{b^4} r_{00}\\
  &+\frac{r^k{ }_k}{b^2}r_{00}+\frac{2 n}{b^2} r_{k 0} s^k{ }_0+\frac{1}{b^2} b^k r_{00 \mid k}+\frac{1}{b^4} r_0^2-\frac{1}{b^2} r_{0 \mid 0}-\frac{2(2 n-3)}{b^4} r_0 s_0+\frac{n-2}{b^2} s_{0 \mid 0}\\
  &-\frac{n-2}{b^4} s_0^2-\frac{1}{b^2 \beta} r_{k 0} s^k\alpha^2-\frac{1}{b^2 \beta} r_k s^k{ }_0\alpha^2-\frac{1}{b^4 \beta} r s_0 \alpha^2 +\frac{1}{b^2\beta} r^k{ }_k s_0\alpha^2+\frac{n-1}{b^2 \beta}s^k{ }_0 s_k \alpha^2 \\
  &-\frac{1}{\beta} s^k{ }_{0 \mid k}\alpha^2+\frac{1}{b^2 \beta} b^k s_{0 \mid k}\alpha^2-\frac{1}{4 \beta^2} s^j{ }_k s^k{ }_j\alpha^4-\frac{1}{2 b^2 \beta^2} s^k s_k\alpha^4.
\end{aligned}$$
\end{Lemma}


\begin{Lemma}[\cite{XQL}]\label{S=0}
Let $F$ be a Kropina metric on $n$-dimensional $M$. Then the followings are equivalent:

(\romannumeral1) $F$ has an isotropic $S$-curvature, i.e., $S=(n+1)cF$;

(\romannumeral2) $r_{00}=\sigma\alpha^2$;

(\romannumeral3) $S=0$;

(\romannumeral4) $\beta$ is a conformal form with respect to $\alpha$,\\
where $c=c(x)$ and $\sigma=\sigma(x)$ are functions on $M$.
\end{Lemma}

\section{Ricci curvature tensor and scalar curvature tensor of Kropina metrics}

By the definition of Ricci curvature tensor and Lemma \ref{Ric}, we obtain the Ricci curvature tensor of Kropina metrics.
\allowdisplaybreaks[2]
\begin{Proposition}
Let $F$ be a Kropina metric on an $n$-dimensional manifold $M$. Then the Ricci curvature tensor of $F$ is given by
\begin{align*}
Ric_{kl}&={}^\alpha Ric_{kl}+F_{\cdot k\cdot l}\left\{-\frac{n-5}{b^4 F^3}r_{00}^2+\frac{2}{b^4 F^2} r_{00}(s_0-2 r_0) +\frac{1}{b^2 F^2} r_{00|0}+\frac{(n-1)}{2 b^2} s_m s_0^m+\frac{(b^2 r_m^m-r) s_0}{2 b^4}\right.\\
&+\frac{(n+1)}{b^2 F} r_{0m} s_0^m+\frac{1}{2 b^2}(s_{0|m}b^m-r_m s_0^m-r_{0 m} s^m)-\frac{1}{2}\left(s_{0|m}^m+\frac{F}{b^2} s_m s^m+\frac{F}{2} s_m^t s_t^m\right)\\
&\left.-\frac{(n+1)}{b^4 F} s_0 r_0\right\}+F_{\cdot l}F_{\cdot k}\left\{\frac{3(n-5)}{b^4 F^4} r_{00}^2+\frac{4}{b^4 F^3} r_{00}(2 r_0-s_0) -\frac{2}{b^2 F^3} r_{00|0}+\frac{n+1}{b^4 F^2} s_0 r_0\right.\\
&\left.-\frac{(n+1)}{b^2 F^2} r_{0 m} s_0^m-\frac{1}{2 b^2} s^m s_m-\frac{1}{4} s_m^t s_t^m\right\}+F_{\cdot k}\left\{-\frac{6(n-3)}{b^4 F^3} r_{00} r_{l 0}+\frac{(n-7)}{b^4 F^2} r_{0l} r_0 -\frac{(n-3)}{b^4 F^2} r_{0l} s_0\right.\\
&+\frac{(n-3)}{b^4 F^2} r_{00} r_l+\frac{2}{b^4 F^2} r_{00} s_l+\frac{2 }{b^2 F^2}r_{l0|0}-\frac{(n-1)}{2 b^2 F^2} r_{00|l}+\frac{(n-1)}{2 b^2} s_m s_l^m+\frac{(b^2 r_m^m-r)}{2 b^4} s_l\\
 &-\frac{r_m s_l^m}{2 b^2}-\frac{r_l^m s_m}{2 b^2}-\frac{s_{l |m}^m}{2}+\frac{s_{l|m} b^m}{2 b^2}\left.-\frac{(n+1)}{2 b^4 F}(s_l r_0+s_0 r_l)+\frac{(n+1)}{2 b^2 F}(r_{lm} s_0^m+r_{0m} s_l^m)\right\} \\
 & +F_{\cdot l}\left\{-\frac{6(n-3)}{b^4 F^3} r_{00} r_{k 0}+\frac{(n-7)}{b^4 F^2} r_{k0} r_0-\frac{(n-3)}{b^4 F^2} r_{k0} s_0+\frac{(n-3)}{b^4 F^2} r_{00} r_k+\frac{2}{b^4 F^2} r_{00} s_k\right.\\
 &+\frac{2}{b^2 F^2} r_{k0|0}-\frac{(n-1)}{2 b^2 F^2} r_{00|k}+\frac{(n-1)}{2 b^2} s_m s_k^m+\frac{(b^2 r_m^m-r)}{2 b^4}s_k -\frac{r_m s_k^m}{2 b^2}-\frac{r_{k m} s^m}{2 b^2}-\frac{s_{k|m}^m}{2}\\
 &+\frac{s_{k|m} b^m}{2 b^2}\left.-\frac{(n+1)}{2 b^4 F}(s_k r_0+s_0 r_k)+\frac{(n+1)}{2 b^2 F}(r_{k m} s_0^m+r_{0 m} s_k^m)\right\}+\frac{8(n-2)}{b^4 F^2} r_{k 0} r_{l0}\\
 &+\frac{4(n-2)}{b^4 F^2} r_{00} r_{kl}+\frac{2(n-1)}{b^4 F} r_{kl} s_0-\frac{2(n-3)}{b^4 F} r_{k l} r_0-\frac{(3n-5)}{b^4 F}(r_{l0} r_k+r_{k0} r_l)\\
 &+\frac{(n-3)}{b^4 F}(r_{k0} s_l+r_{l0} s_k)-\frac{(3n-7)}{2 b^4}(r_k s_l+r_l s_k)-\frac{(n-2)}{b^4}s_k s_l+\frac{r_k r_l}{b^4}-\frac{2 r_{kl|0}}{b^2 F}\\
 &+\frac{(n-1)}{b^2 F}(r_{k 0|l}+r_{l 0|k})+\frac{n-2}{2 b^2}(s_{k|l}+s_{l|k})+\frac{(b^2 r_m^m-r)}{b^4} r_{k l}+\frac{b^m r_{kl|m}}{b^2}-\frac{1}{2 b^2}(r_{k|l}+r_{l|k})\\
 &+\frac{n-1}{2 b^2}(r_{km} s_l^m+r_{lm} s_k^m),
\end{align*}
where ${}^\alpha Ric_{kl}$ denotes the Ricci curvature tensor of $\alpha$.
\end{Proposition}

Contracting the Ricci curvature tensor with $g^{kl}$, we can get the expression of the scalar curvature $R$ of Kropina metrics as following.

\begin{Proposition}
Let $F$ be a Kropina metric on an $n$-dimensional manifold $M$. Then the scalar curvature of $F$ is given by
 \begin{align}\label{R}
 R&=-\frac{24(n-2)\beta}{b^6F^5}r_{00}^2\\
 &+\frac{1}{F^4}\left\{-\frac{(n-1)(n-8)}{b^4}r_{00}^2+\frac{40(n-2)\beta}{b^6}r_{00}r_0
  -\frac{8(n-2)\beta}{b^6}r_{00}s_0-\frac{4(n-2)\beta}{b^4}r_{00|0}\right\}\nonumber\\
 &+\frac{1}{F^3}\left\{\frac{2(n-1)}{b^2}r_{00|0}+\frac{4(n-1)}{b^4}r_{00}(s_0-2r_0)
 +\frac{2(n-2)\beta}{b^4}r_{0|0}-\frac{2(n-2)\beta}{b^4}s_{0|0}\right.\nonumber\\
 &+\frac{2(n-2)\beta}{b^4}b^mr_{00|m}+\frac{2(n-1)\beta}{b^4}r^k_0r_{k0}-\frac{4(n-2)\beta}{b^4}r_{k0}s^k_0-\frac{14(n-2)\beta}{b^6}r_0^2
 +\frac{2(n-2)\beta}{b^6}s_0^2\nonumber\\
 &\left.+\frac{2(n-3)(b^2r^m_m-r)\beta}{b^6}r_{00}+\frac{12(n-2)\beta}{b^6}r_0s_0-\frac{2(3n-5)\beta}{b^6}r_{00}r-\frac{2\beta}{b^2}{}^\alpha Ric\right\}\nonumber\\
 &+\frac{1}{F^2}\left\{-\frac{(n^2+4n-7)}{b^4}s_0r_0+\frac{(n^2+2n-1)}{b^2}r_{0m}s^m_0
 -\frac{(n-5)(b^2r^m_m-r)\beta}{b^6}r_0+\frac{3(n-1)\beta}{b^6}rr_0\right.\nonumber\\
 &+\frac{(n-1)(b^2r^m_m-r)}{b^6}\beta s_0-\frac{(2n-3)\beta}{b^4}r^m_0s_m
 -\frac{(3n-7)\beta}{b^6}rs_0-\frac{(n-2)}{b^4}s_0^2+\frac{r_0^2}{b^4}-\frac{\beta}{b^2}r^m_{m|0}\nonumber\\
 &+\frac{(n-1)}{b^2}\beta r^m_{0|m}-\frac{(2n-1)\beta}{b^4}r_{m0}r^m+\frac{(n-2)\beta}{b^4}b^ms_{0|m}-\frac{(n-2)\beta}{b^4}s_ms^m_0
 +\frac{(n-2)}{b^2}s_{0|0}\nonumber\\
 &+\frac{(b^2r^m_m-r)}{b^4}r_{00}-\frac{(n-2)\beta}{b^4}b^mr_{0|m}
 +\frac{b^mr_{00|m}}{b^2}-\frac{r_{0|0}}{b^2}+\frac{(n-2)\beta}{b^4}r_ms^m_0+{}^\alpha Ric\nonumber\\
 & \left.+\frac{\beta}{b^2}(b^ky^l+b^ly^k){}^\alpha Ric_{kl}\right\}+\frac{1}{F}\left\{\frac{(n^2-1)}{2b^2}s_ms^m_0+\frac{(n+1)(b^2r^m_m-r)}{2b^4}s_0+\frac{(n+1)s_{0|m}b^m}{2b^2}\right.\nonumber\\
 &-\frac{(n+1)}{2b^2}r_ms^m_0-\frac{(n+1)}{2b^2}r_{0m}s^m-\frac{(n+1)}{2}s^m_{0|m}-\frac{(n-2)\beta}{b^4}s_ms^m+\frac{\beta r^mr_m}{b^4}-\frac{(n-3)\beta}{2b^4}r^ks_k\nonumber\\
 &+\frac{(n-2)\beta}{2b^2}s^m_{|m}+\frac{r^m_{m}r^t_{t}}{2b^2}\beta-\frac{rr^m_m\beta}{b^4}+\frac{\beta}{2b^2}b^mr^k_{k|m}-\frac{\beta}{2b^2}r^m_{|m}
 +\frac{(n-1)\beta}{2b^2}r^k_ms^m_k\nonumber\\
 &\left.+\frac{\beta}{2 b^2}(b^2{}^\alpha R-b^kb^l{}^\alpha Ric_{kl})\right\}-\frac{n}{2b^2}s_ms^m-\frac{n}{4}s^t_ms^m_t\nonumber
\end{align}
where ${}^\alpha R$ denotes the scalar curvature of $\alpha$.
\end{Proposition}

\section{The proof of main theorem}
In this section, we will prove Theorem \ref{miR}.

\begin{proof}
``Necessity".
Assume Kropina metric $F$ is of scalar curvature, i.e., $R=n(n-1)\kappa(x)$.  Substituting (\ref{R}) into $R=n(n-1)\kappa(x)$ yields
\begin{align}\label{first}
\alpha^{10}\Gamma_0+\alpha^{8}\Gamma_1+\alpha^{6}\Gamma_2+\alpha^{4}\Gamma_3+\alpha^{2}\Gamma_4+\Gamma_5=0,
\end{align}
where 
\begin{align*}
\Gamma_0&=-n\left[(n-1)\kappa(x)+\frac{1}{2 b^2}s^m s_m+\frac{1}{4} s_{\;t}^m s_{\;m}^t\right],\\
\Gamma_1&=\left[\frac{(n^2-1)}{2b^2}s_ms^m_0+\frac{(n+1)(b^2r^m_m-r)}{2b^4}s_0+\frac{(n+1)s_{0|m}b^m}{2b^2}-\frac{(n+1)}{2b^2}r_ms^m_0-\frac{(n+1)}{2b^2}r_{0m}s^m\right.\\
 &\left.-\frac{(n+1)}{2}s^m_{0|m}\right]\beta+\left[\frac{1}{2 b^2}(b^2{}^\alpha R-b^kb^l{}^\alpha Ric_{kl})-\frac{(n-2)}{b^4}s_ms^m+\frac{r^mr_m}{b^4}-\frac{(n-3)}{2b^4}r^ks_k\right.\\
 &\left.+\frac{(n-2)}{2b^2}s^m_{|m}+\frac{r^m_{m}r^t_{t}}{2b^2}-\frac{rr^m_m}{b^4}+\frac{1}{2b^2}b^mr^k_{k|m}-\frac{1}{2b^2}r^m_{|m}
 +\frac{(n-1)}{2b^2}r^k_ms^m_k\right]\beta^2,\\
\Gamma_2&=\left[{}^\alpha Ric+\frac{(b^2r^m_m-r)}{b^4}r_{00}+\frac{(n^2+2n-1)}{b^2}r_{0m}s^m_0-\frac{(n^2+4n-7)}{b^4}s_0r_0+\frac{r_0^2}{b^4}\right.\\
&\left.-\frac{(n-2)}{b^4}s_0^2+\frac{(n-2)}{b^2}s_{0|0} +\frac{b^mr_{00|m}}{b^2}-\frac{r_{0|0}}{b^2}\right]\beta^2+\left[\frac{(n-2)}{b^4}r_ms^m_0-\frac{(2n-3)}{b^4}r^m_0s_m\right.\\
 &-\frac{(n-5)(b^2r^m_m-r)}{b^6}r_0+\frac{3(n-1)}{b^6}rr_0+\frac{(n-1)(b^2r^m_m-r)}{b^6} s_0 -\frac{(3n-7)}{b^6}rs_0\\
 &-\frac{(n-2)}{b^4}s_ms^m_0-\frac{(2n-1)}{b^4}r_{m0}r^m-\frac{1}{b^2}r^m_{m|0}-\frac{(n-2)}{b^4}b^mr_{0|m}+\frac{(n-1)}{b^2} r^m_{0|m}\\
 &\left.+\frac{(n-2)}{b^4}b^ms_{0|m}+\frac{2}{b^2}b^ky^l{}^\alpha Ric_{kl}\right]\beta^3,\\
\Gamma_3&=\left[\frac{2(n-1)}{b^2}r_{00|0}+\frac{4(n-1)}{b^4}r_{00}(s_0-2r_0)\right]\beta^3
 +\left[\frac{2(n-2)}{b^4}r_{0|0}-\frac{2(n-2)}{b^4}s_{0|0}\right.\\
 &+\frac{2(n-2)}{b^4}b^mr_{00|m}+\frac{2(n-1)}{b^4}r^k_0r_{k0}-\frac{4(n-2)}{b^4}r_{k0}s^k_0-\frac{14(n-2)}{b^6}r_0^2
 +\frac{2(n-2)}{b^6}s_0^2\\
 &\left.+\frac{2(n-3)(b^2r^m_m-r)}{b^6}r_{00}+\frac{12(n-2)}{b^6}r_0s_0-\frac{2(3n-5)}{b^6}r_{00}r-\frac{2}{b^2}{}^\alpha Ric\right]\beta^4,\\
\Gamma_4&=\left[-\frac{(n-1)(n-8)}{b^4}r_{00}^2+\frac{40(n-2)\beta}{b^6}r_{00}r_0
  -\frac{8(n-2)\beta}{b^6}r_{00}s_0-\frac{4(n-2)\beta}{b^4}r_{00|0}\right]\beta^4,\\
\Gamma_5&=-\frac{24(n-2)}{b^6}r_{00}^2\beta^6.
\end{align*}

By (\ref{first}), we have that $\alpha^2$ divides $\Gamma_5$. Thus there exists a scalar function $c=c(x)$ such that $r_{00}=c\alpha^2$, which is the second formula of (\ref{condition1}). Thus we deduce that
\begin{align*}
 &r_{ij}=ca_{ij};\ r_{i0}=cy_i;\ r_{ij|m}=c_ma_{ij};\ r_{i0|m}=c_my_i;\ r_{i0|0}=c_0y_i;\\
 &r_{00|k}=c_k\alpha^2;\ r_{00|0}=c_0\alpha^2;\ r^k_{\;k}=nc;\ r^k_{\;k|0}=nc_0;\ r_{i}=cb_i;\\
 &r_{0}=c\beta;\ r_{i|j}=c_jb_i+cs_{ij}+c^2a_{ij};\ r=cb^2;\ r_{i|0}=c_0b_i+cs_{i0}+c^2y_i;\\
 &r_{0|j}=c_j\beta+cs_{0j}+c^2y_j;\ r_{0|0}=c_0\beta+c^2\alpha^2;\ r^k_{\;|k}=c_b+nc^2,
\end{align*}
where $c_i=\frac{\partial c}{\partial x^i}$, $c_b=c_i b^i$,  $c_0=c_i y^i$.

Substituting the above equations into (\ref{first}) yields
\begin{align}\label{second}
\alpha^{6}\Delta_0+\alpha^{4}\Delta_1+\alpha^{2}\Delta_2+\Delta_3=0,
\end{align}
where 
\begin{align*}
\begin{split}
\Delta_0&=-n\left[(n-1)\kappa(x)+\frac{1}{2 b^2}s^m s_m+\frac{1}{4} s_{\;t}^m s_{\;m}^t\right],\\
\Delta_1&=\frac{1}{2b^4}\{b^2(b^2{\;}^\alpha R-b^kb^l{\;}^\alpha Ric_{kl})+(n+1)b^2[(n-2)c^2+c_b]\\
&-(n-2)(2s^ms_m-b^2s^m_{\;|m})\}\beta^2+\frac{(n+1)}{2b^2}[(n-1)s_ms^m_0+(n-3)cs_0+s_{0|m}b^m-b^2s^m_{0|m}]\beta,\\
\Delta_2&=\frac{(n-2)}{b^4}(c_b-3c^2)\beta^4+\frac{1}{b^4}\left\{2b^2b^ky^l{\;}^\alpha Ric_{kl}+(n-2)[2b^2c_0+b^m s_{0|m}-s^m_{\;0}s_m-5cs_0]\right\}\beta^3\\
        &+\left[{}^\alpha Ric+\frac{(n-2)}{b^4}(b^2s_{0|0}-s_0^2)\right]\beta^2,\\
\Delta_3&=\frac{2}{b^6}\{-b^4 {\;}^\alpha Ric+(n-2)[c^2\beta^2+(2cs_0-b^2c_0)\beta+(s_0^2-b^2s_{0|0})]\}\beta^4.
\end{split}\end{align*}

By (\ref{second}), we have that $\alpha^2$ divides $\Delta_3$, i.e., there exists a scalar function $f=f(x)$ such that
\begin{align}\label{2-0}
f\alpha^2=-b^4{\;}^\alpha Ric+(n-2)[c^2\beta^2+(2cs_0-b^2c_0)\beta+(s_0^2-b^2s_{0|0})].
\end{align}
Differentiating the above equation with respect to $y^i y^j$ yields
\begin{align*}
2fa_{ij}=&-2b^4 {\;}^\alpha Ric_{ij}+(n-2)[2c^2b_ib_j+2c(s_ib_j+s_jb_i)-b^2(c_ib_j+c_jb_i)\\
          &+2s_is_j-b^2(s_{i|j}+s_{i|j})].
\end{align*}
Contracting this formula with $b^ib^j$ or $a^{ij}$ yields, respectively,
\begin{align}
f&=-b^2b^ib^j {\;}^\alpha Ric_{ij}+(n-2)b^2 c^2-(n-2)b^2 c_b-(n-2)s^m s_m,\label{2-4}\\
nf&=-b^4{\;}^\alpha R+(n-2)b^2 c^2-(n-2)b^2 c_b+(n-2)(s^m s_m-b^2s_{\;|m}^m).
\end{align}
Combining the above two formulas, we get
\begin{align}\label{f}
f=\frac{b^2}{n-1}\left(b^ib^j{\;}^\alpha Ric_{ij}-b^2{\;}^\alpha R \right)+\frac{n-2}{n-1}(2s^ms_m-b^2s^m_{\;|m})
\end{align}
and
\begin{align}\label{2-6}
s^m_{\;|m}=-\frac{1}{n-2}(b^2{\;}^\alpha R-nb^ib^j{\;}^\alpha Ric_{ij})+(n-1)(c_b-c^2)+\frac{(n+1)}{b^2}s^m s_m.
\end{align}

Plugging (\ref{2-0})-(\ref{2-6}) into (\ref{second}), we obtain
\begin{align}\label{third}
\alpha^{2}\Theta_0+\Theta_1=0,
\end{align}
where 
\begin{align*}
\begin{split}
\Theta_0&=-n\left[(n-1)\kappa(x)+\frac{1}{2 b^2}s^m s_m+\frac{1}{4} s_{\;t}^m s_{\;m}^t\right],\\
\Theta_1&=\frac{(n+1)}{2b^4}\beta\left\{\left[b^2 b^kb^l{\;}^\alpha Ric_{kl}+(n-2)s^ms_m+(n-1)b^2c_b\right]\beta+(n-1)b^2s_ms^m_{\;0}\right.\\
&\left.+(n-3)b^2cs_0+b^2b^ms_{0|m}-b^4s^m_{0|m}\right\}.
\end{split}
\end{align*}
By (\ref{third}), we have
\begin{equation}(n-1)\kappa(x)=-\frac{1}{4b^2}(2s^ms_m+b^2s^t_{\;m}s^m_{\;t})\end{equation}\label{6}
and
\begin{equation}\left[b^kb^l{\;}^\alpha Ric_{kl}+(n-2)\frac{s^ms_m}{b^2}+(n-1)c_b\right]\beta+(n-1)s_ms^m_{\;0}+(n-3)cs_0+b^ms_{0|m}-b^2s^m_{0|m}=0.\end{equation}\label{6-0}
(\ref{6}) means that $R=-\frac{n}{4b^2}(2s^ms_m+b^2s^t_{\;m}s^m_{\;t})$.
Differentiating (\ref{6-0}) with respect to $y^i$ and contracting it with $b^i$ yield
\begin{equation}(n-1)c_b+b^kb^l{}^\alpha Ric_{kl}+s^m_{\;|m}+s^t_{\;m}s^m_{\;t}=0\end{equation}\label{6-1}
substituting (\ref{2-6}) into (\ref{6-1}) yields $$(n-1)(2c_b-c^2)+\frac{2(n-1)}{n-2}b^kb^l{}^\alpha Ric_{kl}-\frac{b^2}{n-2}{}^\alpha R+s^t_{\;m}s^m_{\;t}+(n+1)\frac{s^ms_m}{b^2}=0.$$

``Sufficiency". It is obviously true.

This completes the proof of Theorem \ref{miR}.
\end{proof}

\section{Other related results}
In this section, we consider the case where $s_0=0$ in Theorem \ref{miR}. 
\begin{Corollary}\label{s_0}
Let $F$ be a Kropina metric on an $n$$(\geq 3)$-dimensional manifold $M$. Assume $s_0=0$. Then $F$ is of isotropic scalar curvature if and only if
\begin{equation}\label{s0=0}
\left\{
\begin{array}{lcl}
\displaystyle b^kb^l{}^\alpha Ric_{kl}=\frac{b^2}{2(n-1)}{}^\alpha R-\frac{(n-2)}{2(n-1)}s^t_{\;m}s^m_{\;t}-\frac{(n-2)}{2}(2c_b-c^2),\vspace{1ex}\\
\displaystyle r_{00}=c(x)\alpha^2,\vspace{1ex}\\
\displaystyle f\alpha^2=-b^4{\;}^\alpha Ric+(n-2)[c^2\beta^2-b^2c_0\beta],
\end{array}\right.
\end{equation}
In this case, $R=-\frac{n}{4}s^t_{\;m}s^m_{\;t}$.
\end{Corollary}

Based on Lemma \ref{S=0} and Theorem \ref{miR}, we get the following result.
\begin{Theorem}\label{miS}
Let Kropina metric $F$ be of isotropic scalar curvature. Then $F$ is of isotropic $S$-curvature if and only if $S$=0.
\end{Theorem}
\begin{proof}
Assume that $F$ is of isotropic scalar curvature. By Theorem \ref{miR}, we know that $r_{00}=c\alpha^2$. By Lemma \ref{S=0},  the result is obviously true.
\end{proof}


 \begin{Lemma}[\cite{L-S}]\label{chi=0}
    For a Finsler metric or a spray on a manifold $M$, $R^{{\ }m}_{i\;mj}=R^{{\ }m}_{j\;mi}$ if and only if $\chi_i=0$.
 \end{Lemma}
\begin{Remark}
Li-Shen defined $\chi=\chi_idx^i$ with the $S$-curvature in \cite{LiShen}, where $\chi_i:=\frac{1}{2}\{S_{.i|m}y^m-S_{|i}\}.$ Based on  Theorem \ref{miS}, we know that $\chi_i$ of Kropina metric with isotropic scalar curvature vanishes, i.e., $R^m_{i\;mj}=R^m_{j\;mi}$. This means that $Ric_{ij}=\overline{Ric}_{ij}$.
\end{Remark}

\section{Conclusions}
In this paper, we study the Kropina metric with isotropic scalar curvature. Firstly, we obtain the expressions of Ricci curvature tensor and scalar curvature. Based on these, we characterize Kropina metrics with isotropic scalar curvature by tensor analysis in Theorem \ref{miR}. In Corollary \ref{s_0}, we discuss the case where $s_0=0$. Kropina metrics with isotropic scalar curvature are deserved further study by the navigation method.

\vspace{6pt}




Liulin Liu

  College of Mathematics and Systems Science,

  Xinjiang University,

  Urumqi, Xinjiang Province, 830046, China,

  Email: 107552000443@stu.xju.edu.cn\\

Xiaoling Zhang

  College of Mathematics and Systems Science,

  Xinjiang University,

  Urumqi, Xinjiang Province, 830046, China,

  Email: zhangxiaoling@xju.edu.cn\\
  
Lili Zhao
  
  School of Mathematical Sciences, 
  
  Shanghai Jiao Tong University, 
  
  Shanghai 200240, China,
  
  Email: zhaolili@sjtu.edu.cn\\

\end{document}